\documentclass[12pt]{article}
\usepackage{amssymb, amsmath}
\usepackage{subcaption}
\usepackage{amsthm}
\usepackage{mathrsfs}
\usepackage{hyperref}
\usepackage[margin=1 in]{geometry}
\usepackage{enumerate}
\allowdisplaybreaks[1]
\numberwithin{equation}{section}

\newtheorem{thm}{Theorem}[section]

\newtheorem{cor}{Corollary}[section]

\newtheorem{rem}{Remark}[section]
\newtheorem{pro}{Proposition}[section]
\newtheorem{lemma}{Lemma}[section]
\usepackage{graphicx}
\begin{document}
	\markboth{R. Rajkumar and T. Anitha}{}
	\title{Some results on the reduced power graph of a group}
	
	\author {R. Rajkumar\footnote{e-mail: {\tt rrajmaths@yahoo.co.in}}~ and T. Anitha\footnote{e-mail: {\tt tanitha.maths@gmail.com}, }
		\\ \small \it Department of Mathematics,
		\small \it The Gandhigram Rural Institute-Deemed to be University,\\
		\small \it Gandhigram--624 302, Tamil Nadu, India.\\
		}
\date{}
	\maketitle
	
\begin{abstract}
	
The  reduced power graph  $\mathcal{RP}(G)$ of a group $G$ is the graph with vertex set $G$ and two vertices $u$ and $v$ are adjacent if and only if  $\left\langle v\right\rangle \subset \left\langle u \right\rangle $ or $\left\langle u\right\rangle \subset \left\langle v \right\rangle $. The proper reduced power graph $\mathcal{RP}^*(G)$ of  $G$ is the subgraph of $\mathcal{RP}(G)$ induced on $G\setminus \{e\}$. In this paper,  
we classify the finite groups whose reduced power graph (resp. proper reduced power graph) is one of complete $k$-partite, acyclic, triangle free or claw-free (resp. complete $k$-partite, acyclic, triangle free, claw-free, star or tree). In addition, we obtain the clique number and the chromatic number of $\mathcal{RP}(G)$ and $\mathcal{RP}^*(G)$ for any torsion group $G$. Also, for a finite group $G$, we determine the girth of $\mathcal{RP}^*(G)$.  Further, we discuss the cut vertices, cut edges and perfectness of these graphs. Then we investigate the connectivity, the independence number and the Hamiltonicity of the reduced power graph
(resp. proper reduced power graph) of some class of groups. Finally, we determine the number of components and the diameter of $\mathcal{RP}^*(G)$ for any finite group $G$. 

\end{abstract}

\section{Introduction}
Algebraic graph theory deals with the interplay between algebra and graph theory.  The investigation of the algebraic properties of groups or rings by using the study of a suitably associated graph is a useful technique. One of the graph that has attracted considerable attention of the researchers is the power graph associated with a group (semigroup). Kelarev and Quinn in  \cite{Kelarev DPG 2002} introduced the directed power graph of a semigroup. Later,  Chakrabarty et al. in \cite{Chakrabarty PG 2009} defined the undirected power graph of a semigroup.  The \textit{directed power graph} $\overrightarrow{\mathcal{P}}(S)$ of a semigroup $S$, is a digraph with vertex set $S$, and for $u,v\in S$, there is an arc from $u$ to $v$ if and only if $u\neq v$ and $v=u^n$ for some integer $n$, or equivalently  $u\neq v$ and $\langle v \rangle \subseteq \langle u \rangle$. The \textit{undirected power graph} $\mathcal{P}(S)$ of $S$  is the underlying graph of $\overrightarrow{\mathcal{P}}(S)$.  Given a group $G$, the \textit{directed proper power graph}  $\overrightarrow{\mathcal{P}}^*(G)$  \textit{of $G$} (resp. \textit{(undirected) proper power graph} $\mathcal{P}^*(G)$ \textit{of $G$}) is defined as the subdigraph of $\overrightarrow{\mathcal{P}}(G)$ (resp. the subgraph of $\mathcal{P}(G)$) induced on $G \setminus \{e\}$, where $e$ denotes the identity element of $G$.  Lots of research has been made on the study of power graphs of groups and semigroups and the results obtained have shown the significance of these graphs.  We refer the reader to  the survey paper \cite{survey PG} for the detailed research on the  power graphs  associated to groups and semigroups and \cite{curtin 2015,doostabadi PG 2015,doost,moghaddamfar PG 2014,ramesh 2018, Xma} for some recent results in this direction.

In order to reduce the complexity of the arcs and edges involved in the above graphs, in \cite{raj RPG 2017}, the authors defined the following graphs: Let $S$ be a semigroup. The \textit{directed reduced power graph} $\overrightarrow{\mathcal{RP}}(S)$  \textit{of $S$}  is the digraph with vertex set $S$, and for $u,v\in S$, there is an arc from $u$ to $v$ if and only if  $v=u^n$ for some integer $n$ and $\langle v \rangle \neq \langle u \rangle$, or equivalently  $\langle v \rangle \subset \langle u \rangle$. The \textit{(undirected) reduced power graph} $\mathcal{RP}(S)$ \textit{of $S$}  is the underlying graph of $\overrightarrow{\mathcal{RP}}(S)$. That is, the vertex set of $\mathcal{RP}(S)$ is $S$ and two vertices $u$ and $v$ are adjacent if and only if   $\langle v \rangle \subset \langle u \rangle$ or $\langle u \rangle \subset \langle v \rangle$. Given a group $G$,  the \textit{directed proper reduced power graph} $\overrightarrow{\mathcal{RP}}^*(G)$ \textit{of $G$} (resp. \textit{(undirected) proper reduced power graph $\mathcal{RP}^*(G)$ \textit{of $G$}}) is defined as the subdigraph of $\overrightarrow{\mathcal{RP}}(G)$ (resp. the subgraph of $\mathcal{RP}(G)$) induced on $G \setminus \{e\}$.

Clearly, $\overrightarrow{\mathcal{RP}}(G)$ is a spanning subdigraph of $\overrightarrow{\mathcal{P}}(G)$.  It is proved in \cite{raj RPG 2017} that the directed reduced power graph  of a group $G$ can determine its set of elements orders and consequently, some class of groups were determined by their directed reduced power graphs. This shows the importance of these graphs to the theory of groups. Moreover, some results on the characterization of groups whose  reduced power graph (resp. proper reduced power graph) is one of the following: complete graph, path, star, tree, bipartite, triangle-free (resp. complete graph, totally disconnected, path, star) were also  established.


The main objective of this paper is to explore some more graph theoretic properties of the (proper) reduced power graphs  of groups. 

The rest of the paper is organized as follows: 
In Section 2,  we explore the structure of the reduced power graph and the proper reduced power graph of a group. We show that the reduced power graph of a torsion group (resp. proper reduced power graph) is $(\Omega(n)+1)$--partite (resp. $\Omega(n)$--partite). In addition, we obtain the clique number, and the chromatic number of the reduced power graph (resp. proper reduced power graph) of any torsion group. As a consequence of this,  we show that among all the reduced power graphs (resp. proper reduced power graphs) of finite groups of given order $n$, $\mathcal{RP}(\mathbb Z_n)$ (resp. $\mathcal{RP}^*(\mathbb Z_n))$ attains the maximum clique number and the maximum chromatic number. Then we classify the groups whose reduced power graphs (resp. proper reduced power graphs) is one of complete $k$-partite or triangle-free (resp. complete k-partite, or triangle-free, star or tree). Also, for a finite group $G$, we determine the girth of $\mathcal{RP}^*(G)$.   

In Section 3, we show that for any finite group $G$, $\mathcal{RP}(G)$ and $\mathcal{RP}^*(G)$ are perfect and not a cycle. Also we investigate the acyclicity and claw-freeness of $\mathcal{RP}(G)$ and $\mathcal{RP}^*(G)$ for a finite group $G$.

In Section 4, we discuss the cut vertices and cut edges of $\mathcal{RP}(G)$ and $\mathcal{RP}^*(G)$ for a finite group $G$.

In Section 5, we estimate the connectivity and the independence number of reduced power graph
(resp. proper reduced power graph) of the following groups: finite cyclic groups, dihedral groups, quaternion groups,
semi-dihedral groups and finite $p$-groups. Also we study the Hamiltonicity of these graphs.

In Section 6, we determine the number of components of proper reduced power graph of a finite group and express this in terms of the number of components of proper power graph of the corresponding group. Finally, we obtain the diameter of the (proper) reduced power graph of any finite group.

Though out this paper, we denote the set of all prime numbers by $\mathbb P$. We use the standard notations and terminologies of graph theory following~\cite{balakrishnan GT book,Harary 1969}.

\section{$k$-partiteness, Clique number, Chromatic number, Girth}

For a given graph $\Gamma$,  $deg_{\Gamma}(v)$ denotes the degree of  a vertex $v$ in $\Gamma$. The girth $gr(\Gamma)$ of $\Gamma$ is the length of shortest cycle in $\Gamma$, if it exist;
otherwise, we define $gr(\Gamma) = \infty$. The clique number $\omega(\Gamma)$ of $\Gamma$ is the maximum size of a complete graph in $\Gamma$. The chromatic number $\chi(\Gamma)$ of $\Gamma$ is the minimum number of colors needed to color the vertices of $\Gamma$ such that no two adjacent vertices gets the same color. If $\omega(\Gamma)=\chi(\Gamma)$, then $\Gamma$ is said to be weakly $\chi$-perfect. $K_n$ denotes the complete graph on $n$ vertices. $P_n$ and $C_n$, respectively denotes the path and cycle of length $n$. The complete $r$-partite graph having partition sizes $n_1,n_2,\ldots, n_r$ is denoted by $K_{n_1,n_2,\ldots, n_r}$.

First, we start to investigate the  structure of the proper reduced power graph of a group. 

Let $G$ be a torsion group. The order of an element $x$ in  $G$ is denoted by $o(x)$.
Note that, if two vertices $x$ and $y$ in $\mathcal{RP}^*(G)$ (resp. $\mathcal{RP}(G)$) are adjacent, then $o(x)\mid o(y)$ or $o(y)\mid o(x)$ according as $\langle x \rangle \subset \langle y \rangle$ or $\langle y \rangle \subset \langle x \rangle$. We frequently use the contrapositive of this to show the non adjacency of two vertices in  $\mathcal{RP}^*(G)$ (resp. $\mathcal{RP}(G)$). 
Let $\pi_e(G)$ denotes the set of elements orders of $G$. Let $n\in \pi_e{(G)}$ be such that $\Omega(n)$ is maximum, where $\Omega(n)$ denotes the number of prime factors of $n$ counted with multiplicity. For each $i=1,2,\ldots,\Omega(n)$,
let
\begin{align}\label{partion name}
X_i:=\{x\in G\mid \Omega(o(x))=i\},
\end{align}	 Clearly, $\{X_i\}_{i=1}^{\Omega(n)}$ is a partition of the vertex set of $\mathcal{RP}^*(G)$. Also for any $x_1,x_2\in X_i$, either $o(x_1)=o(x_2)$ or $o(x_1)$ and $o(x_2)$ does not divide each other. Consequently, no two elements in the same partition are  adjacent in $\mathcal{RP}^*(G)$. Let $x\in G$ be such that $\Omega(o(x))=\Omega(n)$. Let $o(x)=p_1^{\alpha_1}p_2^{\alpha_2}\ldots p_k^{\alpha_k}$, where $p_i$'s are distinct primes, $\alpha_i\geq 1$ for all $i$. Then for each $i=1,2,\ldots,\Omega(n)$, $X_i$ has the elements in $\left\langle x\right\rangle $ whose order is $p_1^{\beta_1}p_2^{\beta_2}\ldots p_k^{\beta_k}$ with $0\leq \beta_j\leq \alpha_j$ and $\beta_1+\beta_2+\cdots +\beta_k=i$. For each $i=1,2,\ldots,\Omega(n)$, we choose $x_i\in \left\langle x \right\rangle $ as follows: First, we fix an element $x_1$ in $X_1$ of order $q_1$, where $q_1\in \{p_1,p_2,\ldots,p_k\}$. If $\Omega(n)=1$, then we stop the procedure. If $\Omega(n)\geq 2$, then we choose an element $x_2$ in $X_2$ of order $q_1q_2$, where $q_2\in \{p_1,p_2,\ldots,p_k\}$ such that $q_1q_2$ divides $o(x)$. Proceeding in this way, at the $i^{th}$ stage $(i=3,4,\ldots, \Omega(n))$, we choose an element $x_i$ in $X_i$ of order $q_1q_2\ldots q_i$, where $q_i\in \{p_1p_2\ldots,p_k\}$ such that $q_1q_2\ldots q_i$ divides $o(x)$. Then $\left\langle x_r\right\rangle \subset \left\langle x_s\right\rangle $, where $1\leq r<s\leq \Omega(n)$ and so $x_1,x_2,\ldots,x_{\Omega(n)}$ forms a clique in  $\mathcal{RP}^*(G)$. Also, since $\Omega(n)$ is maximum, it follows that $\Omega(n)$ is the minimal number such that $\mathcal{RP}^*(G)$ is a $\Omega(n)$-partite graph and so $\omega(\mathcal{RP}^*(G))=\Omega(n)$. This implies that $\chi(\mathcal{RP}^*(G))=\Omega(n)$.

The above procedure is illustrated in Figure~\ref{dihedral_24}, which   shows the structure of $\mathcal{RP}^*(D_{24})$, where $a$ and $b$ are the generators of $D_{24}$.
\begin{figure}[ht!]
	\begin{center}
		\includegraphics [scale=0.8] {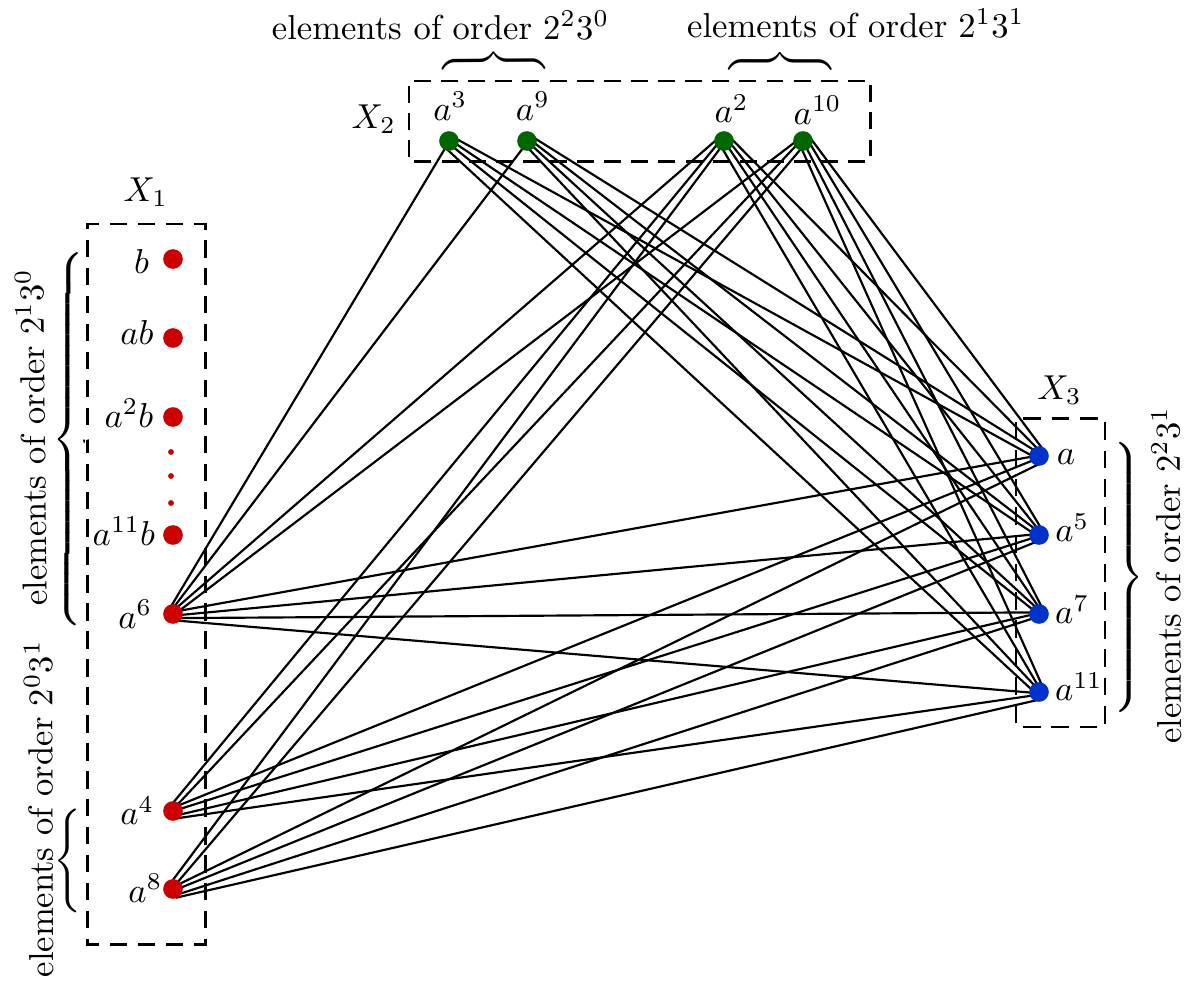}
		\caption{The structure of $\mathcal{RP}^*(D_{24})$} \label{dihedral_24}
	\end{center}
\end{figure}

The above facts are summarized in the following result.
\begin{thm}\label{partite}
	Let $G$ be a torsion group and let $n\in \pi_e{(G)}$ be such that $\Omega(n)$ is maximum. Then
	\begin{enumerate}[\normalfont(1)]
		\item $\mathcal{RP}^*(G)$ is $\Omega(n)$--partite;
		\item $\omega(\mathcal{RP}^*(G))=\Omega(n)=\chi(\mathcal{RP}^*(G))$; and so $\mathcal{RP}^*(G)$ is weakly $\chi$-perfect.	\end{enumerate}
\end{thm}

\begin{cor}\label{partite2}
	Let $G$ be a torsion group and let $n\in \pi_e{(G)}$ be such that $\Omega(n)$ is maximum. Then
	\begin{enumerate}[\normalfont(1)]
		\item $\mathcal{RP}(G)$ is $\Omega(n)+1$--partite;
		\item $\omega(\mathcal{RP}(G))=\Omega(n)+1=\chi(\mathcal{RP}(G))$; and so $\mathcal{RP}(G)$ is weakly $\chi$-perfect.
	\end{enumerate}
\end{cor}

In the next result, we show that among all the reduced power graphs (resp. proper reduced power graphs) of finite groups of given order $n$, $\mathcal{RP}(\mathbb Z_n)$ (resp. $\mathcal{RP}^*(\mathbb Z_n))$ attains the maximum clique number and the maximum chromatic number.
\begin{cor}
	Let $G$ be a non-cyclic group of order $n$. Then $\omega(\mathcal{RP}(G))< \omega(\mathcal{RP}(\mathbb Z_n))$ and $\chi(\mathcal{RP}(G))< \chi(\mathcal{RP}(\mathbb Z_n))$. Also $\omega(\mathcal{RP}^*(G))< \omega(\mathcal{RP}^*(\mathbb Z_n))$ and $\chi(\mathcal{RP}^*(G))< \chi(\mathcal{RP}^*(\mathbb Z_n))$.
\end{cor}

\begin{proof}
	Let $m \in \pi_e(G)$ be such that $\Omega(m)$ is maximum. Then $m\mid n$ and since $G$ is non-cyclic, which forces $m\neq n$. This implies that $\Omega(m)< \Omega(n)$. By using this fact together with part (2) of Theorem~\ref{partite} and Corollary~\ref{partite2}, we get the result.
\end{proof}

The classification of finite groups for which the reduced power graph is given in \cite[Corollary 2.11]{raj RPG 2017}. As a consequence of Theorem~\ref{partite}(1), in the next result, we characterize the torsion groups whose proper reduced power graph is bipartite.

\begin{cor}
	Let $G$ be a torsion group. Then $\mathcal{RP}^*(G)$ is bipartite if and only if $\pi_e(G)$ has a non-prime number which is of the form either $p^2$ or $pq$, where $p$, $q$ are distinct primes.
\end{cor}

The following result is an immediate consequence of Theorem~\ref{partite}(2).

\begin{cor}\label{c3 free}
	Let $G$ be a torsion  group. Then
	
	\begin{enumerate}[\normalfont(1)]
		\item $\mathcal{RP}^*(G)$ is triangle-free if and only if $\Omega(n)\leq 2$ for every $n\in \pi_e(G)$;
		
		\item $\mathcal{RP}(G)$ is triangle-free if and only if every non-trivial element of $G$ is of prime order.
		
\end{enumerate} \end{cor}

In the following results we classify all finite groups whose reduced power graph (resp. proper reduced power graph) is complete $k$-partite. 

\begin{thm}\label{complete partite main thm}
	Let $G$ be a finite group and $p,q$ be distinct primes. Then
	\begin{enumerate}[\normalfont(1)]
		\item $\mathcal{RP}^*(G)$ is complete bipartite if and only if $G\cong \mathbb Z_{p^2}$, $\mathbb Z_{pq}$ or $Q_8$;
		
		\item $\mathcal{RP}^*(G)$ is complete $k$-partite, where $k\geq 3$ if and only if $G\cong \mathbb Z_{p^k}$.
	\end{enumerate}
\end{thm}

\begin{proof}
	By Theorem~\ref{partite}, $\mathcal{RP}^*(G)$ is $k$-partite, where $k=\Omega(n)$, $n\in \pi_e{(G)}$ is such that $\Omega(n)$ is maximum.
	Let $|G|$ has $t$ distinct prime divisors.
	
	\noindent\textbf{Case 1.}
	Let $t=1$.
	
	Then $|G|=p^n$, where $p$ is a prime, $n\geq 1$. Then $G$ has a maximal cyclic subgroup, say $H$, of order $p^k$ and so for each $i=1,2,\ldots, k$, the partition $X_i$ defined in~\eqref{partion name} has all the elements in $G$ of order $p^i$. If $G$ has more than one cyclic subgroup of order $p^i$ for some $i<k$, then no non-trivial element in at least one of these subgroups is adjacent with any non-trivial element in $H$ and so $\mathcal{RP}^*(G)$ is not complete $k$-partite. If $G$ has a unique cyclic subgroup of order $p^i$ for each $i=1,2,\ldots, k-1$, then by ~\cite[Theorem 12.5.2]{Marshall}, $G \cong \mathbb Z_{p^k}$ $(k\geq 2)$ or $Q_8$. If $G\cong \mathbb Z_{p^k}$, then $|X_i|=p^{i-1}(p-1)$ for each $i$. So
	\begin{align}\label{str of pm}
	\mathcal{RP}^*(\mathbb Z_{p^k})\cong K_{p-1,p(p-1),\ldots, p^{k-1}(p-1)}
	\end{align} Clearly, $\mathcal{RP}^*(Q_8)\cong K_{1,6}$
	
	\noindent\textbf{Case 2.} Let $t=2$.
	
	Then $|G|=p_1^{n_1}p_2^{n_2}$, where $p_1$, $p_2$ are distinct primes, $n_1,n_2\geq 1$.		
	If $G$ has an element of order $p_i^2$ for some $i=1,2$, then the elements in $X_2$ of order $p_i^2$ are not adjacent to any element in $X_1$ of order $p_j$ $(j\neq i)$ in $\mathcal{RP}^*(G)$, so $\mathcal{RP}^*(G)$ is not complete $k$-partite. If every non-trivial element in $G$ is of prime order, then by~\cite[Theorem 2.9]{raj RPG 2017}, $\mathcal{RP}^*(G)$ is totally disconnected, and so $\mathcal{RP}^*(G)$ is not complete $k$-partite.  Now we assume that $G$  has an element of order $p_1p_2$. If $G$ has more than one cyclic subgroup of any one of the orders $p_1$, $p_2$ or $p_1p_2$, then $\mathcal{RP}^*(G)$ is not complete $k$-partite. If $G$ has a unique subgroup of each of the orders $p_1$, $p_2$ and $p_1p_2$, then $G\cong \mathbb Z_{p_1p_2}$ and in this case $\mathcal{RP}^*(\mathbb Z_{p_1p_2})= K_{p_1+p_2-2,(p_1-1)(p_2-1)}$.
	
	\noindent\textbf{Case 3.} Let $t\geq 3$.
	
	If $G$ has an element of order $p_ip_jp_l$, where $p_i$, $p_j$, $p_l$ are distinct prime divisors of $|G|$, then the elements in $X_2$ of order $p_ip_j$ are not adjacent to the elements in $X_1$ of order $p_l$ in $\mathcal{RP}^*(G)$, so $\mathcal{RP}^*(G)$ is not complete $k$-partite. If $G$ has an element of order $p_ip_j$, then the elements in $X_2$ of order $p_ip_j$ are not adjacent to the elements in $X_1$ of order $p_r$, where $r\neq i,j$ in $\mathcal{RP}^*(G)$, so $\mathcal{RP}^*(G)$ is not complete $k$-partite. If either $G$ has an element of order $p_i^2$ for some $i$ or every non-trivial element in $G$ of prime order, then by the argument used in Case 2, $\mathcal{RP}^*(G)$ is not complete $k$-partite.
	
	Proof follows by combining all the above cases.
\end{proof}

\begin{cor}\label{complete partite}
	Let $G$ be a finite group. Then $\mathcal{RP}(G)$ is complete $k$-partite, where $k\geq 3$ if and only if $G\cong \mathbb Z_{p^{k-1}}$.
\end{cor}

The following three results are established in \cite{moghaddamfar PG 2014,raj RPG 2017}:

\begin{thm}(\cite[Corollary 4.1]{moghaddamfar PG 2014})\label{con PG}
	Let $G$ be a finite $p$-group, where $p$ is a prime. Then $\mathcal{P}^*(G)$ is connected if and only if $G$ is either cyclic or generalized quaternion.
\end{thm}

\begin{thm}(\cite[Theorem 2.15]{raj RPG 2017})\label{connected RPG}
	Let $G$ be a finite group of non-prime order or an infinite group. Then $\mathcal{RP}^*(G)$ is connected if and only if  $\mathcal{P}^*(G)$ is connected.
\end{thm}

\begin{lemma}(\cite[Lemma 2.16]{raj RPG 2017})\label{pendent vertex}
	If $G$ is a group, and $x$ is a pendent vertex in $\mathcal{RP}^*(G)$, then $o(x)=4$. Converse is true if  $\left\langle x \right\rangle$ is  either $\mathbb Z_4$ or a maximal cyclic subgroup of $G$ when $G$ is noncyclic.
\end{lemma}

The next result is an immediate consequence of Theorems~\ref{connected RPG} and \ref{con PG}.

\begin{cor}\label{connected p group}
	Let $G$ be a finite $p$-group, where $p$ is a prime. Then $\mathcal{RP}^*(G)$ is connected if and only if $G$ is either cyclic or quaternion. 
\end{cor}

\begin{thm}\label{star tree}
	Let $G$ be a finite group. Then the following are equivalent:
	\begin{enumerate}[\normalfont(1)]		
		\item $\mathcal{RP}^*(G)$ is a tree,
		\item $\mathcal{RP}^*(G)$ is a star,
		\item $G\cong \mathbb Z_4$ or $Q_8$.
	\end{enumerate}
\end{thm}

\begin{proof}
	$(2)\Longleftrightarrow  (3)$ is proved in \cite[Theorem 2.17]{raj RPG 2017}. Clearly, $(3)\Rightarrow (1)$. Now we prove $(1) \Rightarrow (3)$. Let $|G|=p_1^{n_1}p_2^{n_2}\ldots p_k^{n_k}$, where $p_i$'s are distinct primes and $k \geq 1$, $n_i\geq 1$ for all $i=1,2,\ldots, k$.
	Since $\mathcal{RP}^*(G)$ is a tree, $\mathcal{RP}^*(G)$ has at least two pendent vertices.  Then $G$ has an element of order 4, by~Lemma\ref{pendent vertex} and this implies that $p_i=2$ and $n_i\geq 2$ for some $i\in \{1,2,\ldots, k\}$. Without loss of generality, we assume that $p_1=2$ and $n_1\geq 2$.
	
	\noindent \textbf{Case 1.} Let $\Omega (n)\geq 3$ for some $n\in \pi_e(G)$.
	
	Then by part (1) of Corollary~\ref{c3 free}, $C_3$ is a subgraph of $\mathcal{RP}^*(G)$.
	
	\noindent \textbf{Case 2.} Let for every $n\in \pi_e(G)$, $\Omega(n)\leq 2$.
	
	If $G$ has an element, say $a$, of order $p_i^2$ ($i\neq 1$) (resp. $p_rp_s(r\neq s)$), then the elements of order $p_i^2~(\text{resp.~} p_rp_s)$ in $\left\langle a \right\rangle $ together with the elements of order $p_i$ (resp. $p_r$ and $p_s$) in $\left\langle a \right\rangle $ forms $C_4$ as a subgraph in $\mathcal{RP}^*(G)$.
	
	\noindent \textbf{Case 3.} Let $\pi_e(G)=\{1,4\}\cup \mathbb{P}$. 
	
	If $k\geq 2$, then the elements of order $p_i$ $(i\neq 1)$ are isolated vertices in $\mathcal{RP}^*(G)$, a contradiction to our assumption that $\mathcal{RP}^*(G)$ is a tree.	
	If $k=1$, then $|G|=2^n$ with $\pi_e(G)=\{1,2,4\}$. Then by Corollary~\ref{connected p group}, $G\cong \mathbb Z_4$ or $Q_8$, since $\mathcal{RP}^*(G)$ is connected. Also, $\mathcal{RP}^*(G)\cong K_{1,2}$ or $K_{1,6}$, which are trees. 
\end{proof}

The girth of the reduced power graph of a finite group is determined by the authors in \cite[Corollary 2.12]{raj RPG 2017}. In the following result, we determine the girth of the proper reduced power graph of a finite group.

\begin{thm}
	Let $G$ be a finite group and let $n\in \pi_e(G)$ be such that $\Omega(n)$ is maximum. Then
	
	$gr(\mathcal{RP}^*(G))=\begin{cases}
	3, & \text{if}~~\Omega(n)\geq 3;\\
	4, & \text{if}~~\Omega(n)=2 ~\text{and}~ \pi_e(G)\neq \{1,4\}\cup \mathbb{P}; \\\infty, & \text{otherwise}.
	\end{cases}$
\end{thm}

\begin{proof}
	If $\Omega(n)\geq 3$, then by Corollary \ref{c3 free}, $gr(\mathcal{RP}^*(G))=3$. If $\Omega(n)=2$ and $\pi_e(G)\neq \{1,4\}\cup  \mathbb P$, then by Theorem~\ref{partite}, $\mathcal{RP}^*(G)$ is bipartite  and $G$ has an element of order $p^2$ ($p >2$) or $pq$, where $p,q$ are distinct prime factors of $|G|$. Consequently, $\mathcal{RP}^*(G)$ has $C_4$ as a subgraph. But the bipartiteness of $G$ forces that $gr(\mathcal{RP}^*(G))=4$. If $\Omega(n)=2$ and $\pi_e(G)= \{1,4\}\cup  \mathbb P$, then $\mathcal{RP}^*(G)$ is the union (not necessarily disjoint) of some copies of the graph $K_{1,2}$ and isolated vertices. It follows that $gr(\mathcal{RP}^*(G))=\infty$. If $\Omega(n)=1$, then $\mathcal{RP}^*(G)$ is totally disconnected, so $gr(\mathcal{RP}^*(G))=\infty$. 
\end{proof}

\section{Perfectness, Acyclicity, Claw-freeness}

The complement of a graph $\Gamma$ is denoted by $\overline \Gamma$.  The join of two graphs $\Gamma_1$ and $\Gamma_2$ is denoted by $\Gamma_1+\Gamma_2$.  A graph
$\Gamma$ is perfect if for every induced subgraph $H$ of $\Gamma$, the chromatic number of $H$ equals the size of the largest clique of $H$. The strong perfect graph theorem \cite{chudnovsky perfect graph thm} states that \textit{a graph is perfect if and only if neither the graph nor its complement contains an odd cycle of length at least 5 as an induced subgraph}. In the following two results we show that $\mathcal{RP}(G)$ and $\mathcal{RP}^*(G)$ are perfect for any finite group $G$.   

\begin{thm}\label{perfect}
	For a finite group $G$, $\mathcal{RP}(G)$ is perfect.
\end{thm}

\begin{proof}
	If we show that neither $\mathcal{RP}(G)$ nor its complement has an induced cycle of odd length at least 5, then by strong perfect graph theorem \cite{chudnovsky perfect graph thm}, $\mathcal{RP}(G)$ is perfect. Assume the contrary. Let $C:u_1-u_2-\cdots-u_{2n+1}-u_1$ be an induced cycle of odd length at least 5 in $\mathcal{RP}(G)$. Then $\left\langle u_i \right\rangle \neq \left\langle u_j \right\rangle $, $i\neq j$. This implies that $C$ is also an induced cycle in $\mathcal{P}(G)$, which is a contradiction to \cite[Theorem 5]{doostabadi PG 2015}, which states that  $\mathcal{P}(G)$ is perfect . Suppose $C': v_1-v_2-\cdots-v_{2n+1}-v_1$ be an induced cycle of odd length at least 5 in $\overline{\mathcal{RP}(G)}$. Then $\left\langle v_i \right\rangle \neq \left\langle v_j \right\rangle $, $i\neq j$. Consequently, $C'$ is also an induced cycle of odd length at least 5 in $\overline{\mathcal{P}(G)}$, which is again a contradiction to the fact that $\mathcal{P}(G)$ is perfect.  
\end{proof}

\begin{cor}
	For a finite group $G$, $\mathcal{RP}^*(G)$ is perfect.
\end{cor}

\begin{proof}
	The proof follows from the definition of perfect graph,  Theorem~\ref{perfect} and the fact that $\mathcal{RP}(G)=K_1+\mathcal{RP}^*(G)$. 
\end{proof}

\begin{pro}\label{acyclic}
	Let $G$ be a finite group and $p$, $q$ be distinct primes. Then we have the following:
	
	\begin{enumerate}[\normalfont(1)]
		\item $\mathcal{RP}^*(G)$ is acyclic if and only if $\pi_e(G)\subseteq \{1,4\}\cup \mathbb{P}$.
		\item $\mathcal{RP}(G)$ is acyclic if and only if $G$ is either a $p$-group with exponent $p$ or non-nilpotent group of order $p^nq$ ($n\geq 1$) with all non-trivial elements are of order $p$ or $q$.
	\end{enumerate} 
\end{pro}

\begin{proof}
	(1) If $\pi_e(G)\nsubseteq \{1,4\}\cup \mathbb P$, then by cases (1) and (2) of Theorem~\ref{star tree}, $\mathcal{RP}^*(G)$ has a cycle. Otherwise, $\mathcal{RP}^*(G)$ is a disjoint union of some copies of star graphs and isolated vertices. So $\mathcal{RP}^*(G)$ is acyclic.
	
	(2) The proof follows from Corollary~\ref{c3 free}(2) and by the classification of finite groups with non-trivial
	elements are of prime order given in~\cite{cheng prime elt}. 
\end{proof}

\begin{pro}\label{abelian acyclic}
	Let $G$ be a finite abelian group and $p$ be a prime. Then we have the following:
	
	\begin{enumerate}[\normalfont(1)]
		\item $\mathcal{RP}^*(G)$ is acyclic if and only if $G\cong \mathbb Z_p^n$, $\mathbb Z_4^n$ or $\mathbb Z_4^m\times \mathbb Z_2^n$, where $n$, $m\geq 1$.
		
		\item $\mathcal{RP}(G)$ is acyclic if and only if $G\cong \mathbb Z_p^n$, where $n\geq 1$.
		
	\end{enumerate} 
\end{pro}

\begin{proof}
	
	(1) By Proposition~\ref{acyclic}(1), it is enough to determine the finite abelian groups whose elements order set is $\{1,4\}\cup \mathbb{P}$. Assume that $\pi_e(G)\subseteq \{1,4\}\cup  \mathbb P$. If $G$ has elements of order $p_1$ and $p_2$, where $p_1$, $p_2$ are two distinct prime divisors of $|G|$, then $G$ has an element of order $p_1p_2$. Consequently, $\pi_e(G)=\{1,p\}$ or $\{1,2,4\}$. In the former case we get $G\cong \mathbb Z_p^n$, $n\geq 1$ and in the later case we get $G\cong \mathbb Z_4^n$ or $\mathbb Z_4^m \times \mathbb Z_2^n$, where $n$, $m\geq 1$.
	
	(2) By Proposition~\ref{acyclic}(2), $\pi_e(G)\subseteq \mathbb{P}$. If the non-trivial elements in $G$ are of prime order, then $G\cong \mathbb Z_p^n$ and so $\mathcal{RP}(G)\cong K_{1,p^n-1}$, which is acyclic.  
\end{proof}

\begin{thm}\label{claw free1}
	Let $G$ be a finite group. Then $\mathcal{RP}(G)$ is claw-free if and only if $G\cong \mathbb Z_n$, $n=2, 3, 4$.
\end{thm}

\begin{proof}
	If $G$ has an element $a$ with $o(a)\geq 5$, then $\mathcal{RP}(\left\langle a \right\rangle )$ has $K_{1,3}$ as an induced subgraph of $\mathcal{RP}(G)$. So we assume that the order of any element in $G$ is at most 4. Then $|G|$ is one of the following: $2^n$, $3^m$ or $2^n3^m$, $n,m\geq 1$. Assume that $|G|=2^n3^m$, $n,m\geq 1$. Then the elements in $G$ of order 2 and 3 together with the identity element in $G$ forms $K_{1,3}$ as an induced subgraph of $\mathcal{RP}(G)$. Next we assume that $|G|=p^n$, $p=2$ or 3 and $n\geq 1$. 
	
	If $G$ has more than one cyclic subgroup of order $p$, then the elements of order $p$ together with the identity element in $G$ forms $K_{1,3}$  as an induced subgraph of $\mathcal{RP}(G)$.
	
	If $G$ has a unique subgroup of order $p$, then $G\cong \mathbb Z_2$, $\mathbb Z_3$, $\mathbb Z_4$ or $Q_8$. If $G\cong Q_8$, then its  elements  of order 4 together with the element of order 2 forms $K_{1,3}$ as an induced subgraph of $\mathcal{RP}(G)$. In the remaining cases, we have $\mathcal{RP}(\mathbb Z_2)\cong K_2$,    $\mathcal{RP}(\mathbb Z_3)\cong K_{1,2}$ and  $\mathcal{RP}(\mathbb Z_4)\cong K_1+K_{1,2}$, which are claw-free.  
\end{proof}

\begin{thm}\label{claw free}
	Let $G$ be a finite group. Then $\mathcal{RP}^*(G)$ is claw free if and only if $\pi_e(G)\subseteq \{1,4\}\cup \mathbb{P}$ and any two cyclic subgroups of order 4 in $G$ intersect trivially.
\end{thm}
\begin{proof}
	Let $a$ be a non-trivial element in $G$. If $o(a)=p^2$, where $p>2$ is a prime, then the elements of order $p$ and $p^2$ in $\left\langle a \right\rangle $ are  adjacent with each other in $\mathcal{RP}^*(G)$ and so $\mathcal{RP}^*(G)$ has $K_{p-1,p(p-1)}$ as an induced subgraph. In particular, it has $K_{1,3}$ as an induced subgraph. If $o(x)=8$, then the element of order 2 together with the elements of order 8 in $\left\langle a \right\rangle $ forms $K_{1,3}$ as an induced subgraph of $\mathcal{RP}^*(G)$. Similarly, if $o(a)=pq$, where $p$, $q$ are distinct primes, then the elements of order $pq$ together with the elements of order $p$ and $q$ in $\left\langle a \right\rangle $ forms $K_{1,3}$ as an induced subgraph of $\mathcal{RP}^*(G)$. Now we assume that  $\pi_e(G)\subseteq \{1,4\}\cup \mathbb{P}$. Let $H$ and $K$ be cyclic subgroups of $G$ of order 4. If they intersect non-trivially, then the non-trivial elements in $H\cup K$ forms $K_{1,3}$ as an induced subgraph of $\mathcal{RP}^*(G)$. If any two cyclic subgroups of order 4 in $G$ intersect trivially, then $\mathcal{RP}^*(G)$ is the disjoint union of the graph $K_{1,2}$ and isolated vertices  and so $\mathcal{RP}^*(G)$ is claw free. 
\end{proof}

The following result shows that a cycle can not be realized as the reduced power graph (resp. proper reduced power graph) of any finite group.

\begin{cor}
	Let $G$ be a finite group. Then $\mathcal{RP}(G)\ncong C_n$ and $\mathcal{RP}^*(G)\ncong C_n$ for any $n\geq 3$.
\end{cor}

\begin{proof} If $\mathcal{RP}(G)\cong C_n$, $n\geq 3$, then $deg_{\mathcal{RP}(G)}(a)=2$ for each non-trivial element $a$ in $G$. Also $deg_{\mathcal{RP}(G)}(e)=|G|-1$. It follows that $|G|=3$. Then $G\cong \mathbb Z_3$ and so $\mathcal{RP}(G)\cong P_3$, which is a contradiction to our assumption. 
	If $\mathcal{RP}^*(G)\cong C_n$, $n\geq 3$, then it is claw-free. So by the proof of Theorem~\ref{claw free}, $\mathcal{RP}^*(G)$ is the disjoint union of the graph $K_{1,2}$ and isolated vertices, which is a contradiction to our assumption that $\mathcal{RP}^*(G)$ is a cycle.  
\end{proof}

\section{Cut vertices and Cut edges}

A cut vertex (cut edge) of a graph is a vertex (edge) whose removal increases the number of components of the graph. A connected graph $\Gamma$ is said to be $k-$connected if there is no set of $k-1$ vertices whose removal disconnects $\Gamma$. 

Note that for any group $G$, if $x$ is a cut vertex of $\mathcal{P}^*(G)$, then $x$ is a cut vertex of $\mathcal{RP}^*(G)$, since $\mathcal{RP}^*(G)$ is a subgraph of $\mathcal{P}^*(G)$. But the converse of this statement is not true. For example, $2$ is a cut vertex of $\mathcal{RP}^*(\mathbb Z_4)$, whereas $\mathcal{P}^*(\mathbb Z_4)$ has no cut vertex. Notice that the degree of 2 in $\mathcal{RP}^*(\mathbb Z_4)$ is 2.  In the next theorem, we show that the degree of a vertex in $\mathcal{RP}^*(G)$ is at least three is a necessary condition to hold the converse of this statement for that vertex.

\begin{pro}\label{cut}
	Let $G$ be a finite group and let $x$ be an element in $G$ with $deg_{\mathcal{RP}^*(G)}(x)\geq 3$. Then $x$ is a cut vertex in $\mathcal{RP}^*(G)$ if and only if $x$ is a cut vertex in $\mathcal{P}^*(G)$.
\end{pro}

\begin{proof}
	Assume that $x$ is a cut vertex in $\mathcal{RP}^*(G)$. Then there exist $y$, $z\in G$ such that every $y-z$ path in $\mathcal{RP}^*(G)$ contains $x$. 
	
	\noindent \textbf{Case 1.}	Let $yz$ is an edge in $\mathcal{P}^*(G)$.\\ Then $\left\langle y\right\rangle =\left\langle z\right\rangle $ and $y$, $z$ are pendent vertices in $\mathcal{RP}^*(G)$, so $y-x-z$ is the only path joining $y$, $z$ in $\mathcal{RP}^*(G)$. It follows that both $y$ and $z$ are adjacent only to the vertex $x$ in $\mathcal{P}^*(G)$. Since $deg_{\mathcal{RP}^*(G)}(x)\geq 3$,  $\mathcal{P}^*(G)-\{x\}$ is disconnected and so $x$ is a cut vertex in $\mathcal{P}^*(G)$.
	
	\noindent \textbf{Case 2.}	Let $y$ be non adjacent to $z$ in $\mathcal{P}^*(G)$.\\ Suppose $x$ is not a cut vertex in $\mathcal{P}^*(G)$, then we can choose a shortest $y-z$ path in $\mathcal{P}^*(G)$, which does not contain $x$, let it be $P(y,z)$. Then $P(y,z):=y(=w_1)-w_2-\cdots -w_{n-1}-(w_n=)z$, where $\left\langle w_i\right\rangle \neq \left\langle w_{i+1} \right\rangle $, $i=1,2,\ldots,k$. It follows that $P(y,z)$ is also a path in $\mathcal{RP}^*(G)$ not containing $x$, a contradiction. Hence $x$ is a cut vertex in $\mathcal{P}^*(G)$. 
	
	Converse is clear.  
\end{proof}


%

Note that, for any group $G$, if $\mathcal{RP}(G)$ has a cut vertex, then it must be the identity element in $G$. Converse of this statement is not true. For example, the identity element is not a cut vertex of $\mathcal{RP}(\mathbb Z_4)$, since $\mathcal{RP}(\mathbb Z_4)-\{e\}=K_{1,2}$ . In the following result, we give a necessary condition for a vertex  of $\mathcal{RP}^*(G)$ to be its cut vertex.

\begin{pro}\label{cut1}
	Let $G$ be a finite group. If an element $x$ in $G$ is a cut vertex in $\mathcal{RP}^*(G)$, then $o(x)=2$.
\end{pro}

\begin{proof}
	Let $x$ be a cut vertex in $\mathcal{RP}^*(G)$. Then there exist $y$, $z\in G$ such that every $y-z$ path in $\mathcal{RP}^*(G)$ contains $x$. Let $P(y,z)$ be a shortest $y-z$ path in $\mathcal{RP}^*(G)$. Then $P(y,z):y(=w_1)-w_2-\cdots-w_i-x-w_{i+1}-\cdots-(w_n=)z$, where $\left\langle w_r\right\rangle \neq \left\langle x \right\rangle $, $r=1,2,\ldots,n$. Suppose $o(x)\neq 2$, then there exist $x'\in G$ such that $x\neq x'$ and $\left\langle x' \right\rangle = \left\langle x \right\rangle $. Consequently, $P'(y,z):y(=w_1)-w_2-\cdots-w_i-x'-w_{i+1}-\cdots-(w_n=)z$ is a $y-z$ path in $\mathcal{RP}^*(G)$ not containing $x$, a contradiction to $x$ being a cut vertex of $\mathcal{RP}^*(G)$. Therefore, $o(x)=2$.  
\end{proof}

Since any cut vertex of $\mathcal{P}^*(G)$ is a cut vertex of $\mathcal{RP}^*(G)$, we get the following result as an immediate consequence of the previous theorem.

\begin{cor}\label{power cut vertex}
	Let $G$ be a finite group. If an element $x$ in $G$ is a cut vertex in $\mathcal{P}^*(G)$, then $o(x)=2$.
\end{cor}

\begin{rem}
	Converse of Proposition~\ref{cut1} is not true. For example,  the element 3 in $\mathbb Z_6$ has order 2, but $\mathcal{RP}^*(\mathbb Z_6)-\{3\}=K_{2,2}$ is connected.
\end{rem}

\begin{cor}
	Let $G$ be a group  of non-prime odd order. Then the following are equivalent:
	\begin{enumerate}[\normalfont(1)]
		
		\item 	$\mathcal{RP}^*(G)$ is connected,
		
		\item $\mathcal{RP}^*(G)$ is 2--connected,
		
		\item $\mathcal{P}^*(G)$ is 2--connected.		
	\end{enumerate}
\end{cor}
\begin{proof}
	$(2)\Rightarrow (1)$ is obvious. Now we  prove $(1)\Rightarrow (2)$. Assume that $\mathcal{RP}^*(G)$ is connected. As a consequence of Proposition~\ref{cut1}, $\mathcal{RP}^*(G)$ has no cut vertex and so it is 2--connected. The proof of $(2)\Longleftrightarrow(3)$ follows from Corollary~\ref{power cut vertex} and \cite[Theorem 2.15]{raj RPG 2017}.  
\end{proof}

\begin{pro}
	Let $G$ be a finite cyclic group. Then $\mathcal{RP}^*(G)$ has a cut vertex if and only if $G\cong \mathbb Z_4$.
\end{pro}

\begin{proof}
	Let $G=\left\langle a\right\rangle $. Assume that $\mathcal{RP}^*(G)$ has a cut vertex, say $x$. Then there exist $y$, $z\in G$ such that every $y-z$ path in $\mathcal{RP}^*(G)$ contains $x$ and by Proposition~\ref{cut1}, $o(x)=2$. Suppose both $y$ and $z$ are non-generators of $G$, then $y-a-z$ is a $y-z$ path not containing $x$, which is a contradiction. Suppose $y$ is a generator of $G$ but not $z$, then $y$ and $z$ are adjacent in $\mathcal{RP}^*(G)$, which is a contradiction. Now we assume that both $y$ and $z$ are generators of $G$, then it forces that both $y$ and $z$ are adjacent only to $x$. Consequently, $o(y)=4=o(z)$, since $o(x)=2$. This implies that $G\cong \mathbb Z_4$. Converse is clear. 
\end{proof}

\begin{pro}\label{cut edge}
	Let $G$ be a finite group and let $e'=xy$ be an edge in $\mathcal{RP}^*(G)$. Then $e'$ is a cut edge in $\mathcal{RP}^*(G)$ if and only if at least one of $\left\langle x\right\rangle$ or $\left\langle y\right\rangle$ is either $\mathbb Z_4$ or a maximal cyclic subgroup of $G$ of order 4 when $G$ is noncyclic.
\end{pro}

\begin{proof}
	Suppose $e'=xy$ is a cut edge in $\mathcal{RP}^*(G)$. Then either $x$ or $y$ is a cut vertex in $\mathcal{RP}^*(G)$. Without loss of generality,    we assume that $x$ is a cut vertex in $\mathcal{RP}^*(G)$. By Proposition~\ref{cut1}, $o(x)=2$ and so $o(y)\neq 2$. Again by Proposition~\ref{cut1}, $y$ is not a cut vertex in $\mathcal{RP}^*(G)$. Therefore, $y$ must be a pendent vertex in $\mathcal{RP}^*(G)$. By Lemma~\ref{pendent vertex}, $y$ satisfies the requirements of the theorem.
	
	
	Conversely, assume with out loss of generality that $\left\langle x\right\rangle $ is either $\mathbb Z_4$ or a maximal cyclic subgroup of $G$ of order $4$ when $G$ is noncyclic.
	Then by Lemma~\ref{pendent vertex}, $x$ is a pendent vertex in $\mathcal{RP}^*(G)$ and is adjacent only with $y$.  It follows that $e'=xy$ is a cut edge in $\mathcal{RP}^*(G)$. 
\end{proof}

\begin{cor}
	$\mathcal{RP}^*(\mathbb Z_n)$, $n \geq 2$ has a cut edge if and only if $n=4$.
\end{cor}

\begin{pro}
	Let $G$ be a finite group and let $e'=xy$ be an edge in $\mathcal{RP}(G)$. Then $e'$ is a cut edge in $\mathcal{RP}(G)$ if and only if at least one of $\left\langle x\right\rangle$ or $\left\langle y\right\rangle$ is either $\mathbb Z_p$, where $p$ is a prime or a maximal cyclic subgroup of $G$ of prime order when $G$ is noncyclic.
\end{pro}

\begin{proof}
	The proof is similar to the proof of Proposition~\ref{cut edge} and using the facts that if $\mathcal{RP}(G)$ has a cut vertex, then it must be an identity element in $G$ and any vertex which is adjacent to the identity element generates a maximal cyclic subgroup of prime order.  
\end{proof}

\begin{cor}
	$\mathcal{RP}(\mathbb Z_n)$, $n \geq 1$ has a cut edge if and only if $n=p$, where $p$ is a prime.
\end{cor}

\section{Independence number, Connectivity, Hamiltonicity}
Let $\Gamma$ be a graph. The independence number $\alpha(\Gamma)$ of $\Gamma$ is the number of vertices in a maximum independent set of $\Gamma$. The connectivity $\kappa(\Gamma)$ of $\Gamma$ is the minimum number of vertices, whose removal results in a disconnected or trivial graph. A graph is Hamiltonian if it has a spanning cycle.
It is known (\cite[Theorem 6.3.4]{balakrishnan GT book}) that if 
$\Gamma$ is Hamiltonian, then for every nonempty proper subset $S$ of
$V(\Gamma)$, the number of components of $G-S$ is at most $|S|$. Also, if $\Gamma$ is $2$-connected and $\alpha(\Gamma)\leq \kappa(\Gamma)$, then $\Gamma$ is Hamiltonian  (\cite[Theorem 6.3.13]{balakrishnan GT book}). We use these facts for investigation of the Hamiltonicity of the reduced power graph of groups.

In this section, we explore the independence number, connectivity and Hamiltonicity of  reduced power graph (resp. proper reduced power graph) of cyclic groups, dihedral groups, quaternion groups, semi dihedral groups and finite $p$-groups. Note that for any group $G$,  $\mathcal{RP}(G)=K_1+\mathcal{RP}^*(G)$, so it follows that $\alpha(\mathcal{RP}(G))=\alpha(\mathcal{RP}^*(G))$.

\subsection{Cyclic groups}


\begin{pro}
	Let $n$ be a positive integer and let $\varphi(n)$ denotes its Euler's totient function. Then
	\begin{enumerate}[\normalfont(1)]
		
		\item $\kappa(\mathcal{RP}(\mathbb Z_n))= n-\varphi(n)$, if $2\varphi(n)+1\geq n$;   
		
		\item $\kappa(\mathcal{RP}(\mathbb Z_n))\geq \varphi(n)+1$, if $2\varphi(n)+1 < n$. The equality holds for $n=2p$, where $p$ is a prime. 		
	\end{enumerate}
	
\end{pro}

\begin{proof}
	Let $S=\{g\in \mathbb Z_n\mid o(g)\neq n\}$. Then $|S|=n-\varphi(n)$ and $\mathcal{RP}(\mathbb Z_{n})- S=\overline{K}_{\varphi(n)}$. Therefore, $\kappa(\mathcal{RP}(\mathbb Z_{n}))\leq n-\varphi(n)$. Since  the identity element  and the generators of $\mathbb Z_n$ are adjacent to all the vertices in $\mathcal{RP}(\mathbb Z_n)$. It follows that $\kappa(\mathcal{RP}(\mathbb Z_{n}))\geq \varphi(n)+1$, if $n-\varphi(n)>\varphi(n)+1$; $\kappa(\mathcal{RP}(\mathbb Z_{n}))=n-\varphi(n)$, if $n-\varphi(n)\leq \varphi(n)+1$.
	
	If $n=2p$, where $p$ is a prime, then $\mathcal{RP}(\mathbb Z_{2p})\cong K_1+K_{p,p-1}$. Consequently, $\kappa(\mathcal{RP}(\mathbb Z_{2p}))=p=\varphi(2p)+1$.  
\end{proof}

\begin{cor}
	Let $n \geq 2$ be an integer. Then we have the following:
	\begin{enumerate}[\normalfont(1)]
		
		\item $\kappa(\mathcal{RP}^*(\mathbb Z_n))= n-\varphi(n)-1$ if $2\varphi(n)+1\geq n$;
		
		\item $\kappa(\mathcal{RP}^*(\mathbb Z_n))\geq \varphi(n)$ if $2\varphi(n)+1< n$. The equality holds for $n=2p$, where $p$ is a prime.
	\end{enumerate}
\end{cor}

\begin{cor}\label{kappa}
	$\kappa(\mathcal{RP}(\mathbb Z_{p^m}))=p^{m-1}$ and $\kappa(\mathcal{RP}^*(\mathbb Z_{p^m}))=p^{m-1}-1$, where $p$ is a prime and $m$ is a positive integer.
\end{cor}

\begin{pro}\label{alpha}
	Let  $n \geq 2$ be an integer. Then we have the following:
	\begin{enumerate}[\normalfont(1)]
		\item $\alpha(\mathcal{RP}^*(\mathbb Z_n))=p^{m-1}(p-1)$ if $n=p^m$, where $p$ is a prime and $m$ is a positive integer.
		
		\item $\alpha(\mathcal{RP}^*(\mathbb Z_n))\geq \varphi(n)$ if $n$ is not a prime power. The equality holds for $n=pq$, where $p$, $q$ are distinct primes with $p,q\neq 2$.
	\end{enumerate}
\end{pro}

\begin{proof}
	Clearly by~\eqref{str of pm},
	$\alpha(\mathcal{RP}^*(\mathbb Z_{p^m}))=p^{m-1}(p-1)$. Assume that $n$ is not a prime power. Since the generators of $\mathbb Z_n$ are not adjacent to each other in $\mathcal{RP}^*(\mathbb Z_{n})$, it follows that $\alpha(\mathcal{RP}^*(\mathbb Z_{n}))\geq \varphi(n)$. Finally, if $p$, $q$ are distinct primes with $p,q\neq 2$, then $\mathcal{RP}(\mathbb Z_{pq})=K_1+{K_{\varphi(pq),p+q-1}}$. Since $p,q\neq 2$, $\varphi(pq)=(p-1)(q-1)> p+q-1$, we have $\alpha(\mathcal{RP}^*(\mathbb Z_{pq}))=\varphi(pq)$.  
\end{proof}

\begin{cor}
	Let $n$ be a positive integer. Then we have the following.
	\begin{enumerate}[\normalfont(1)]
		\item $\alpha(\mathcal{RP}(\mathbb Z_n))=p^{m-1}(p-1)$ if $n=p^m$, where $p$ is a prime and $m$ is a positive integer.
		
		\item $\alpha(\mathcal{RP}(\mathbb Z_n))\geq \varphi(n)$ if $n$ is not a prime power. The equality holds for $n=pq$, where $p$, $q$ are distinct primes with $p,q\neq 2$.
	\end{enumerate}
\end{cor}

\begin{pro}\label{cyclic  pgroup hamiltonian} Let $p$ be a prime number and $n\geq 1$. Then we have the following:
	\begin{enumerate}[\normalfont(1)]
		\item
		$\mathcal{RP}(\mathbb Z_{p^n})$  is Hamiltonian if and only if $p=2$ and $n\geq 2$;
		\item  $\mathcal{RP}^*(\mathbb Z_{p^n})$ is non-Hamiltonian.
	\end{enumerate}
\end{pro}

\begin{proof}
	
	\noindent (1) We divide the proof in to two cases.\\ 
	\textbf{Case 1.} Let $p\geq 3$. Take $S=G\setminus X_n$, where $X_n$ is given in~\eqref{partion name}. Then by~\eqref{str of pm}, $\mathcal{RP}(\mathbb Z_{p^n})-S=\overline{K}_{p^{n-1}(p-1)}$. It follows that the number of components of $\mathcal{RP}(\mathbb Z_{p^n})-S$ is $p^{n-1}(p-1)>p^{n-1}=|S|$. Therefore,  $\mathcal{RP}(G)$ is non-Hamiltonian. 
	
	\noindent \textbf{Case 2.} Assume that $p=2$. Clearly, $\mathcal{RP}(\mathbb Z_2)\cong K_2$, which is non-Hamiltonian. 
	Now	let $n\geq 2$. Then by Proposition~\ref{alpha} and Corollary \ref{kappa},  $\alpha(\mathcal{RP}(\mathbb Z_{2^n}))=2^{n-1}= \kappa(\mathcal{RP}(\mathbb Z_{2^n}))$. Also by \eqref{str of pm}, $\mathcal{RP}^*(G)$ is connected. It follows that $\mathcal{RP}(G)$ is 2-connected, since the identity element adjacent to all the elements of  $\mathcal{RP}(G)$. Therefore, $\mathcal{RP}(G)$ is Hamiltonian. 
	
	(2) Proof follows by taking $S=G\setminus (X_n \cup \{e\})$ and following the similar argument as in the proof of part (1).	
\end{proof}

\subsection{Dihedral groups, Quaternion groups and Semi-dihedral groups}


The dihedral group of order $2n$ ($n\geq 3$) is given by $D_{2n}=\left\langle a,b \mid a^n=e=b^2,\right.$ $\left. ab=ba^{-1} \right\rangle $.
\begin{pro}\label{dihedral 1}
	For an integer $n\geq 3$,
	\begin{enumerate}[\normalfont(1)]
		\item $\kappa(\mathcal{RP}(D_{2n}))=1$ and $\kappa(\mathcal{RP}^*(D_{2n}))=\kappa(\mathcal{RP}^*(\mathbb Z_{n}))$;
		\item $\alpha(\mathcal{RP}(D_{2n}))=n+\alpha(\mathcal{RP}(\mathbb Z_n))$;
		
		\item $\mathcal{RP}(D_{2n})$ and $\mathcal{RP}^*(D_{2n})$ are non-Hamiltonian.
	\end{enumerate}
\end{pro}

\begin{proof}
	Since $D_{2n}=\left\langle a\mid a^n=e\right\rangle \cup \{a^ib\mid 1\leq i\leq n\}$, it is easy to see that the structure of
	
	\begin{figure}[ht]
		\begin{center}
			\includegraphics [scale=0.8] {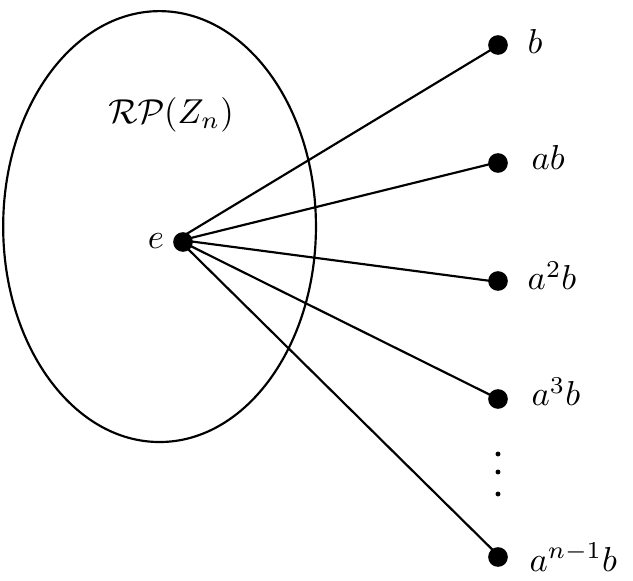}
			\caption{The graph $\mathcal{RP}(D_{2n})$} \label{str dn}
		\end{center}
	\end{figure}
	
	$\mathcal{RP}(D_{2n})$ is as shown in Figure~\ref{str dn} and so
	\begin{align}
	\mathcal{RP}(D_{2n})\cong K_1+\left( \mathcal{RP}^*(\mathbb Z_n)\cup \overline{K}_n \right)\label{str Dn},\\
	\mathcal{RP}^*(D_{2n})=\mathcal{RP}^*(\mathbb Z_{n})\cup \overline{K}_n \label{str1 Dn}
	\end{align} 
	
	By \eqref{str Dn}, we can determine the values of $\kappa(\mathcal{RP}(D_{2n}))$ and $\alpha(\mathcal{RP}(D_{2n}))$. Moreover, $\mathcal{RP}(D_{2n})$ has a cut vertex and so it is non-Hamiltonian.  The proofs for $\mathcal{RP}^*(D_{2n})$ follows from \eqref{str1 Dn}. 
\end{proof}

\noindent The quaternion group of order $4n$ ($n\geq 2$) is given by $Q_{4n}=\left\langle a, b\mid  a^{2n}=e=b^4, \right.$ $\left. bab^{-1}=a^{-1} \right\rangle $.
\begin{pro}\label{Qn nonhamil}
	For an integer $n\geq 2$,
	\begin{enumerate}[\normalfont(1)]
		\item  $\kappa(\mathcal{RP}(Q_{4n}))=2$ and $\kappa(\mathcal{RP}^*(Q_{4n}))=1$;
		
		\item $\alpha(\mathcal{RP}(Q_{4n}))\in \left\lbrace 2n+\alpha(\mathcal{RP}(\mathbb Z_{2n}))-1, 2n+\alpha(\mathcal{RP}(\mathbb Z_{2n})) \right\rbrace $;  
		\item $\mathcal{RP}(Q_{4n})$ and $\mathcal{RP}^*(Q_{4n})$ are non-Hamiltonian.
	\end{enumerate}
\end{pro}

\begin{proof}
	Note that $Q_{4n}=\left\langle a\mid a^{2n}=e\right\rangle \cup \{a^ib\mid 1\leq i\leq 2n\}$. It has $\left\langle a^ib\right\rangle $ as a maximal cyclic subgroup of order $4$ and $\left\langle a^ib \right\rangle \cap \left\langle a^jb \right\rangle =\{e,a^n\}$, $i,j\in \{1,2,\ldots, 2n\}$, for all $i\neq j$. So $a^n$ is adjacent to all the elements of order 4 in $\mathcal{RP}(Q_{4n})$. From these facts, it is easy to see that the structure of  $\mathcal{RP}(Q_{4n})$ is as in Figure~\ref{str qn}. 
	
	\begin{figure}[ht]
		\begin{center}
			\includegraphics [scale=0.8] {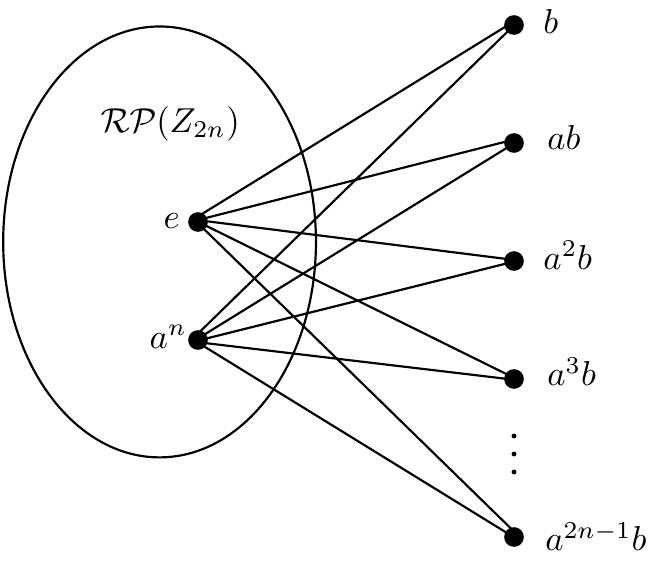}
			\caption{The graph $\mathcal{RP}(Q_{4n})$} \label{str qn}
		\end{center}
	\end{figure}
	So $\mathcal{RP}(Q_{4n})$ has no cut vertex and $\mathcal{RP}(Q_{4n})- \{e,a^n\}$ is disconnected. Therefore, $\kappa(\mathcal{RP}(Q_{4n}))=2$ and $\kappa(\mathcal{RP}^*(Q_{4n}))=1$. It follows that $\mathcal{RP}^*(Q_{4n})$ is non-Hamiltonian. Moreover, the number of components of  $\mathcal{RP}(Q_{4n})-\{e,a^n\}$ is at least $2n+1>2$. So, $\mathcal{RP}(Q_{4n})$ is non-Hamiltonian. It can be seen from Figure~\ref{str qn} that if at least one of a maximal independent set of $\mathcal{RP}(\mathbb Z_{2n })$ does not have the element of order 2 in $\mathbb Z_{2n}$, then $\alpha(\mathcal{RP}(Q_{4n}))=2n+\alpha(\mathcal{RP}(\mathbb Z_{2n}))$; otherwise, $\alpha(\mathcal{RP}(Q_{4n}))=2n+\alpha(\mathcal{RP}(\mathbb Z_{2n}))-1$. 
\end{proof}

\noindent The semi-dihedral group of order $8n$ ($n\geq 2$) is given by $SD_{8n}=\left\langle a, b\mid a^{4n}=e \right.$ $\left.=b^2 bab^{-1}=a^{2n-1} \right\rangle $.
\begin{pro}
	For an integer $n\geq 3$,
	\begin{enumerate}[\normalfont(1)]
		\item  $\kappa(\mathcal{RP}(SD_{8n}))=1= \kappa(\mathcal{RP}^*(SD_{8n}))$;
		
		\item $\alpha(\mathcal{RP}(SD_{8n}))\in \left\lbrace 4n+\alpha(\mathcal{RP}(\mathbb Z_{4n}))-1, 4n+\alpha(\mathcal{RP}(\mathbb Z_{4n}))\right\rbrace $;
		\item  $\mathcal{RP}(SD_{8n})$ and $\mathcal{RP}^*(SD_{8n})$ are non-Hamiltonian.
	\end{enumerate}
\end{pro}

\begin{proof}
	Note that 
	
	$SD_{8n}=\left\langle a\mid a^{4n}=e \right\rangle \cup \{a^kb\mid 1\leq k\leq 4n~\text{and}~ k ~\text{is even}\} \cup \{a^kb\mid 1\leq k\leq 4n~\text{and}~ k~ \text{is odd}\}$. 
	
	\begin{figure}[h!]
		\begin{center}
			\includegraphics [scale=0.8] {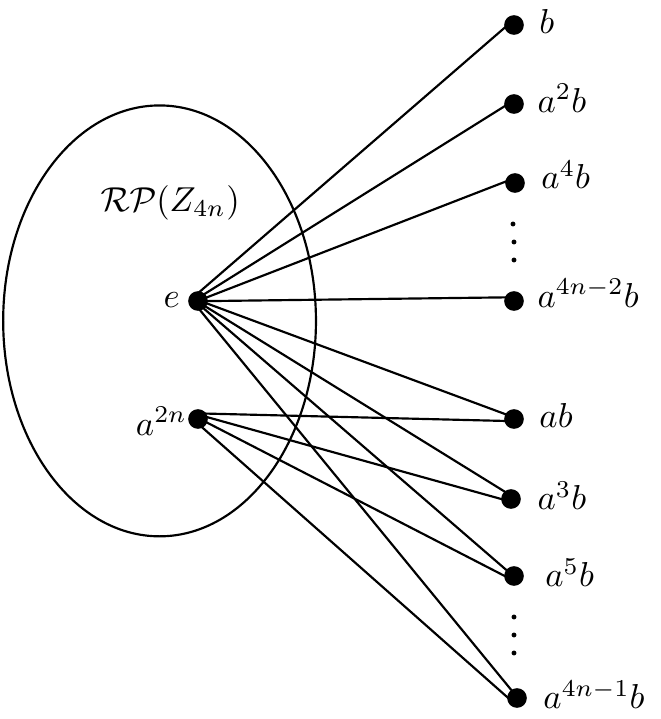}
			\caption{The graph $\mathcal{RP}(SD_{8n})$} \label{str qd}
		\end{center}
	\end{figure}  
	
	Here for each $k$,  $1\leq k\leq 4n$, $k$ is odd,  $\left\langle a^kb\right\rangle $  is a maximal cyclic subgroup of order 4 in $SD_{8n}$ and all subgroups of $SD_{8n}$ have an element of order 2 in $\left\langle a \right\rangle $ as common, so they are adjacent only to the identity element and the element of order 2 in $\left\langle a \right\rangle $; $\left\langle a^kb\right\rangle $ where $k$ is even, $1\leq k\leq 4n$ is maximal cyclic subgroup of order 2, so they are adjacent only the identity element in $SD_{8n}$. Hence the structure of  $\mathcal{RP}(SD_{8n})$ is as shown in Figure~\ref{str qd}.

	Moreover, the identity element is a cut vertex of $\mathcal{RP}(SD_{8n})$ and $a^{2n}$ is a  cut vertex of  $\mathcal{RP}^*(SD_{8n})$ and so these two graphs are non-Hamiltonian and their connectivity is 1. From  Figure~\ref{str qd} it is not hard to see that if at least one of a maximal independent set of $\mathcal{RP}(\mathbb Z_{4n})$ doesn't have the element of order 2, then $\alpha(\mathcal{RP}(SD_{8n}))=4n+\alpha(\mathcal{RP}(\mathbb Z_{4n}))$; otherwise, $\alpha(\mathcal{RP}(SD_{8n}))=4n+\alpha(\mathcal{RP}(\mathbb Z_{4n}))-1$. 
\end{proof}

\subsection{Finite $p$-groups}

\begin{cor}
	Let $G$ be a finite $p$-group, where $p$ is a prime. Then 
	\begin{center}
		$\kappa(\mathcal{RP}(G))=
		\begin{cases}
		p^{m-1}, & \text{if}~G\cong \mathbb Z_{p^m}, m\geq 1;\\
		2,& \text{if}~G\cong Q_{2^\alpha}, \alpha\geq 3;\\
		1, &\text{otherwise}.
		\end{cases}$
	\end{center}
\end{cor}

\begin{proof}
	Proof follows from Proposition~\ref{Qn nonhamil}(1), Corollaries~\ref{connected p group} and~\ref{kappa} 
\end{proof}

\begin{cor}\label{pgroup hamiltonian}
	Let $G$ be a finite $p$-group, where $p$ is a prime. Then
	\begin{enumerate}[\normalfont(1)]
		\item  $\mathcal{RP}(G)$ is Hamiltonian if and only if $G\cong \mathbb Z_{2^n}$, $n\geq 2$;
		\item $\mathcal{RP}^*(G)$ is non-Hamiltonian.
	\end{enumerate}	 
\end{cor}

\begin{proof}
	%
	(1) If $G$ is a non-cyclic and non-quaternion $p$-group, then by Corollary~\ref{connected p group}, $\mathcal{RP}^*(G)$ is non-Hamiltonian. Moreover, the identity element in $G$ is a cut vertex of $\mathcal{RP}(G)$. Hence $\mathcal{RP}(G)$ is non-Hamiltonian. If $G \cong  \mathbb Z_{p^n}$ or  $Q_{2^\alpha}$, where $p$ is a prime,  $n\geq 1$ and $\alpha\geq 3$, then the proof follows by Propositions~\ref{cyclic  pgroup hamiltonian} and \ref{Qn nonhamil}(3).
	
	(2) Proof is similar to the proof of part (1). 
\end{proof}

\section{Number of components and diameter}

A component of a graph $\Gamma$ is its maximal connected subgraph. We denote the number of components of $\Gamma$ by $c(\Gamma)$. The distance $d(u,v)$ between two vertices $u$ and $v$ is the length of the shortest $u-v$ path in $\Gamma$; If there is no path between $u$ and $v$, we set $d(u,v)$ as $\infty$. The diameter $diam(\Gamma)$ of a connected graph $\Gamma$ is defined as $max\{d(u,v)\mid u,v\in \Gamma\}$. If $\Gamma$ is disconnected, then we define $diam(\Gamma) = \infty.$

In the following result, we explore a relationship between the number of components of the proper reduced power graph and the number of components of the proper power graph of a finite group.


\begin{thm}\label{connected}	
	
	Let $G$ be a finite group and $p$ be a prime. Then 	
	\begin{align*}
	c(\mathcal{RP}^*(G))	=	c(\mathcal{P}^*(G))+\displaystyle \sum_{p | |G|} r(p-2),
	\end{align*}

	\noindent where $r=0$, if $G \cong \mathbb Z_n$ with $n \neq p$; $r=1$, if $G \cong \mathbb Z_p$; and $r$ equals the number of maximal cyclic subgroups of prime order $p$ in $G$, otherwise.
\end{thm}

\begin{proof}
	
	We notice that each equivalence class of $\sim$ is entirely contained in the vertex set of a component of $\mathcal{P}^*(G)$, and the vertex set of each component of  $\mathcal{P}^*(G)$  is the union of some equivalence classes of $\sim$.  
	On the other hand, the elements of any equivalence class of $\sim$ are either entirely contained in the vertex set of a component of $\mathcal{RP}^*(G)$ or isolated vertices, and the vertex set of each component of   $\mathcal{RP}^*(G)$ is either an union of some equivalence classes of $\sim$ or consists exactly one isolated vertex.

	Also for each integer $k \geq 2$, there is a one to one correspondence between the set $\mathcal{A}_k$ of all components of $\mathcal{P}^*(G)$  containing  exactly $k$ equivalence class of $\sim$ and, the set $\mathcal{B}_k$  of all components of $\mathcal{RP}^*(G)$  containing  exactly $k$ equivalence class of $\sim$. If we take $$\mathcal{A} := \displaystyle \bigcup_{k\geq 2} \mathcal{A}_k$$  and  $$\displaystyle \mathcal{B} := \bigcup_{k\geq 2}\mathcal{B}_k,$$ then it follows that there is a one to one correspondence between $\mathcal{A}$ and $\mathcal{B}$.
	
	
	
	

	Now let $\mathcal{C}$ denote the  components of  $\mathcal{P}^*(G)$ whose vertex sets contains exactly one equivalence class of $\sim$. Note that
	\begin{align}
	c(\mathcal{P}^*(G))=|\mathcal{A}|+|\mathcal{C}|\label{eqn 2}
	\end{align}
	\noindent \textbf{Case 1.} Let $G$ be cyclic of order $n$.
	
	It is not hard to see that $c(\mathcal{RP}^*(G))=p-1=	c(\mathcal{P}^*(G))+ (p-2)$, if $n=p$ and $c(\mathcal{RP}^*(G))=1=	c(\mathcal{P}^*(G))$,  if $n \neq p$.
	
	\noindent \textbf{Case 2.} Let $G$ be non-cyclic.
	
	If $W \in \mathcal{C}$, then the elements in $W$ must be the non-trivial elements of a cyclic subgroup of order $p$ in $G$, where $p$ is a prime. This cyclic subgroup must be maximal in $G$, since $G$ is non-cyclic and $W$ is a component of $\mathcal{P}^*(G)$. Conversely, each maximal cyclic subgroup of prime order in $G$ corresponds to an element of  $\mathcal{C}$, since $G$ is non-cyclic. Thus, there is a one to one correspondence between $\mathcal{C}$ and the set of all maximal cyclic subgroups of prime order in $G$. Consequently,
	\begin{align}
	|\mathcal{C}| = \displaystyle \sum_{p | |G|} m_p,\label{eqn 1}
	\end{align}
	
	\noindent where $m_p$ denotes the number of maximal cyclic subgroups of prime order $p$ in $G$.

	Notice that the elements in $W$ induces a clique in  $\mathcal{P}^*(G)$	and so they are isolated  in $\mathcal{RP}^*(G)$. 	
	Thus $W$ corresponds to $p-1$ isolated vertices in $\mathcal{RP}^*(G)$. Conversely, each isolated vertex of $\mathcal{RP}^*(G)$ can be obtained in this way. Consequently,
	\begin{align}
	\text{the number of isolated vertices of~} \mathcal{RP}^*(G) = \displaystyle \sum_{p | |G|} m_p(p-1).\label{eqn 3}
	\end{align}

	By using \eqref{eqn 2}, \eqref{eqn 1} and \eqref{eqn 3}, we have
	\begin{eqnarray}
	c(\mathcal{RP}^*(G)) &=& |\mathcal{B}| + \text{ the number of isolated vertices of~} \mathcal{RP}^*(G) \notag\\
	&=& |\mathcal{A}| + \displaystyle \sum_{p | |G|} m_p(p-1) \notag\\
	&=& |\mathcal{A}| + \displaystyle \sum_{p | |G|} m_p + \displaystyle \sum_{p | |G|} m_p(p-2)\notag\\
	&=&	c(\mathcal{P}^*(G))+\displaystyle \sum_{p | |G|} m_p(p-2). \notag
	\end{eqnarray}
	This completes the proof.		
\end{proof}

An immediate consequence of Theorem~\ref{connected} is the following result.
\begin{cor}
	Let $G$ be a finite group and  $p$ be prime. Then $c(\mathcal{RP}^*(G))=c(\mathcal{P}^*(G))$ if and only if either $G\cong \mathbb Z_n$, where $n\neq p$ or $G$ has no maximal cyclic subgroup of order $p\, (\neq 2)$.
\end{cor}

In the next result, we obtain the diameter of the (proper) reduced power graph of a finite group which shows that there is a close relationship between the diameter of the proper power graph of a finite group and the diameter of the proper reduced power graph of that group.

\begin{thm}\label{diam}
	Let $G$ be a finite group and $p$ be a prime. Then we have the following:
	\begin{enumerate}[\normalfont(1)]
		\item  $diam(\mathcal{RP}(G))=\begin{cases}
		1, &\text{if~} G\cong \mathbb Z_2;\\
		2, &\text{otherwise}.
		\end{cases}$
		
		\item  $diam(\mathcal{RP}^*(G))=\begin{cases}
		2, &\text{if~} G\cong \mathbb Z_{p^n}, n\geq 2;\\
		\infty, &\text{if~} G\cong \mathbb Z_{p}.\\
		diam(\mathcal{P}^*(G)), & \text{otherwise~}
		\end{cases}$	
	\end{enumerate}	 
\end{thm}

\begin{proof}
	(1) $diam(\mathcal{RP}(G))=1$ if and only if $G$ is complete; and this holds only when $G\cong \mathbb Z_2$, by \cite[Corollary 2.8]{raj RPG 2017}. If $G$ is other than $\mathbb Z_2$, then $diam(\mathcal{RP}(G))=2$, since the identity element of $G$ is adjacent to all the elements in $\mathcal{RP}(G)$.
	
	(2) Assume that $diam(\mathcal{P}^*(G))=1$. Then $\mathcal{P}^*(G)$ is complete and so by \cite[Corollary 9]{curtin 2015}, $G\cong \mathbb Z_{p^n}$, $n\geq 1$. If $n=1$, then $\mathcal{RP}^*(G)\cong \overline{K}_{p-1}$ and so $diam(\mathcal{RP}^*(G))=\infty$. If $n\geq 2$, then, $\mathcal{RP}^*(G)$ is complete $n$-partite and so $diam(\mathcal{RP}^*(G))=2$. Next, we assume that $diam(\mathcal{P}^*(G))=m$ $(m\geq 2)$. Let $x$, $y$ be vertices of $\mathcal{P}^*(G)$  such that $d_{\mathcal{P}^*(G)}(x,y)=r$, $r>1$. Let $P(x,y):=x-x_1-x_2-\cdots-x_{r-1}-y$ be a shortest $x-y$ path in $\mathcal{P}^*(G)$. Then $\left\langle x_i\right\rangle \neq \left\langle x_j\right\rangle $, for $i\neq j$. So $P(x,y)$ is also a shortest path in $\mathcal{RP}^*(G)$ and so $d_{\mathcal{P}^*(G)}(x,y)=r$.  Therefore, $diam(\mathcal{RP}^*(G))=diam(\mathcal{P}^*(G))$. If $diam(\mathcal{P}^*(G))=\infty$, then $\mathcal{P}^*(G)$ is disconnected. Since $\mathcal{RP}^*(G)$ is subgraph of  $\mathcal{P}^*(G)$, it follows that $diam(\mathcal{RP}^*(G))=\infty$. So $diam(\mathcal{RP}^*(G))=diam(\mathcal{P}^*(G))$, in this case also. 
\end{proof}

The following result is an immediate consequence of Theorem~\ref{diam} and \cite[Theorem 2.8(2)]{doost}.
\begin{cor}\label{diam 2}
	Let $G$ be a finite group. Then the following are equivalent:
	\begin{enumerate}[\normalfont (a)]
		\item 	$diam(\mathcal{RP}^*(G))=2$;
		
		\item $diam(\mathcal{RP}(G))=diam(\mathcal{RP}^*(G))$;
		
		\item $G$ is nilpotent  and its sylow subgroup are cyclic or generalized quaternion 2-group and $G\ncong \mathbb Z_p$, where $p$ is a prime.
	\end{enumerate}	
\end{cor}


%
%


\begin{rem}
	\normalfont A. R. Moghaddamfar and S. Rahbariyan \cite{moghaddamfar PG 2014}, B. Curtin et al  \cite{curtin 2015} and  A. Doostabadi et al  \cite{doost} proved several results on the connectedness and the diameter of the proper power graphs of finite groups. As a consequence of Theorems~\ref{connected} and \ref{diam}(2), the same results are also hold for the proper reduced  power graphs of finite groups, except for the cyclic group of prime order in the case of connectedness, and cyclic $p$-groups in the case of diameter.
	
	Moreover,  the number of components of proper power graphs of some family of finite groups
	including nilpotent groups, groups with a nontrivial partition,  symmetric groups, and
	alternating groups were studied in \cite{doost}. In view of Theorem~\ref{connected}, for a given group, the number of components of its proper power graph may differ from the number of components of  its proper reduced power graph. Finding the number of components of proper reduced power graph of a finite group involves in finding its number of maximal subgroups of prime order. Since the later corresponds to the number of components of its proper power graph containing exactly one equivalence class of $\sim$, so from the process of finding the number of components of proper power graph,  one can easily obtain the number of components of its proper reduced power graph and vice versa.
\end{rem}

\section*{Acknowledgment}
The second author gratefully acknowledges University Grants Commission (UGC), Govt. of India  for granting Rajiv Gandhi National Fellowship for this research work.


\begin{thebibliography}{00}	
	
	\bibitem{survey PG}  J. Abawajy, A. Kelarev and M. Chowdhury, Power graphs: A survey, \emph{Electron. J. Graph Theory Appl.} 1(2) (2013), 125--147.
	
	\bibitem{balakrishnan GT book} R. Balakrishnan and K. Ranganathan, {\it A Text Book of Graph Theory}, Springer Science $\&$ Business Media, New York, 2012.
	
	\bibitem{Chakrabarty PG 2009} I. Chakrabarty, Shamik Ghosh and M.K. Sen, Undirected power graph of semigroup, \emph{Semigroup Forum} 78 (2009), 410--426.
	
	\bibitem{cheng prime elt} Cheng Kai Nah, M. Deaconescu, Lang Mong Lung and Shi Wujie, Corrigendum and addendum to "Classification of finite groups with all elements of prime order", \emph{Proc. Amer. Math. Soc.} 117 (1993), 1205--1207.
	
	
	\bibitem{chudnovsky perfect graph thm}  M. Chudnovsky, N. Robertson, P. Seymour and R. Thomas, The strong perfect graph theorem, \emph{Ann Math.} 164 (2006), 51--229.
	
	
	\bibitem{curtin 2015}  B. Curtin, G.R. Pourgholi and H. Yousefi-Azari, On the punctured power graph of a finite group, \emph{Australas. J. Combin.} 62(1) (2015), 1--7.
	
	\bibitem{doostabadi PG 2015} A. Doostabadi, Erfanian Ahmad and Jafarzadeh Abbas, Some results on the power graphs of finite groups, \emph{Science Asia.} 41 (2015), 73--78.
	
	\bibitem{doost} A. Doostabadi, M. Farrokhi and D. Ghouchan, On the connectivity of proper power graphs of finite groups, \emph{Commun Algebra} 43 (2015), 4305--4319.
	
	
	
	
	
	
	\bibitem{Harary 1969} F.  Harary,  {\it Graph Theory}, Addison-Wesley, Philippines, 1969.
	
	
	
	
	\bibitem{Kelarev DPG 2002} A.V. Kelarev and S. J. Quinn, Directed graph and combinatorial properties of semigroups, \emph{J. Algebra} 251 (2002), 16--26.
	
	%
	
	
	
	\bibitem{Marshall}  Marshall Hall, {\it The Theory of Groups}, AMS Chelsea Publishing, USA, 1999.
	
	\bibitem{moghaddamfar PG 2014} A.R. Moghaddamfar, S. Rahbariyan and W.J. Shi, Certain properties of the power graph associated with a finite group, \emph{J. Algebra Appl.} 13(7) (2014) Article No. 1450040 [18 pages] .
	
	\bibitem{raj RPG 2017} R. Rajkumar and T. Anitha, Reduced power graph of a group,  \emph{Electron. Notes Discrete Math.} 63 (2017), 69--76. 
	
	\bibitem{ramesh 2018} Ramesh Prasad Panda and  K.V. Krishna, On the minimum degree, edgeconnectivity and connectivity of power graphs of finite groups, \emph{Comm. Algebra.} 46(7) (2018), 3182--3197.
	
	
	
	\bibitem{Xma} Xuanlong Ma and Min Feng, On the chromatic number of the power graph of a finite group, \emph{Indag. Math.} 26(4) (2015), 626--633.
	
\end{thebibliography}
\end{document}